\documentclass[12pt,reqno,twoside]{amsart}
\usepackage{amsmath}
\usepackage{amssymb}
\usepackage{epsfig}
\usepackage
{color}

\usepackage[margin=1in]{geometry}

\numberwithin{equation}{section}

\title[Distance-like functions and smooth approximations]{Distance-like functions and smooth approximations \\ \ \\ {\it a correction to} \\ \ \\ ``Logarithm laws for flows on homogeneous spaces"} 
\author{Dmitry Kleinbock}
\address{Brandeis University, Waltham MA
02454-9110 {\tt kleinboc@brandeis.edu}}
\author{Gregory Margulis}
\address{Yale University, New Haven CT 06520
{\tt grigorii.margulis@yale.edu }}
\newif\ifdraft\drafttrue
\draftfalse

\newcommand\eq[2]{{\ifdraft{\ \tt [#1]}\else\ignorespaces\fi}\begin{equation}\label{eq:
#1}{#2}\end{equation}}

\newcommand {\pd}{\Phi_\Delta}

\newcommand {\equ}[1] {\eqref{eq: #1}}

\newcommand{\R}{{\mathbb{R}}}
\newcommand{\Z}{{\mathbb{Z}}}

\newcommand{\N}{{\mathbb{N}}}

\newcommand{\ggm}{G/\Gamma}

\newcommand{\dist}{{\rm dist}}

\newcommand{\df}{{\, \stackrel{\mathrm{def}}{=}\, }}

\newcommand{\vre}{\varepsilon}

\newcommand\hs{homogeneous space}

\newcommand {\ignore}[1] {}

\newtheorem{thm}{Theorem}[section]
\newtheorem{lem}[thm]{Lemma}

\begin{document}
\ignore{}
\begin{abstract}
{One of the propositions in the paper \cite{KM}, 
 related to approximating certain sets by smooth functions, was recently found to be incorrect. Here we 
 correct the mistake.}
\end{abstract}
\thanks{Supported by NSF grants DMS-1265695 and DMS-1600814.}

\date{July 11, 2017}
\maketitle
\section{Statement of results}\label{intro}
\noindent Let us reproduce the setting of \cite{KM} in a slightly more general form. Let $G$ be a 
Lie
group  
and  $\Gamma$ 
a lattice in $G$. Denote by $X$ the \hs\ $\ggm$ and by $\mu$ the $G$-invariant probability measure on $X$. In what follows, $\|\cdot\|_p$ will stand for the $L^p$ norm. Fix a basis $\{Y_1,\dots,Y_n\}$ 
for the Lie algebra
$\frak g$
of
$G$, and, given a smooth
 function $h \in C^\infty(X)$ and $\ell\in\N$, define the ``{\sl $L^2$, order $\ell$" Sobolev norm} $\|h \|_{2,\ell}$ of $h $ by 
 $$
 \|h \|_{2,\ell} \df \sum_{|\alpha| \le \ell}\|D^\alpha h \|_2,
 $$
 where 
 $\alpha = (\alpha_1,\dots,\alpha_n)$ is a multiindex, $|\alpha| = \sum_{i=1}^n\alpha_i$, and $D^\alpha$ is a differential operator of order $|\alpha|$ which is a monomial in  $Y_1,\dots, Y_n$, namely $D^\alpha = Y_1^{\alpha_1}\cdots Y_n^{\alpha_n}$.
This definition depends on the basis,
however, a change of basis
would only  distort
$ \|h \|_{2,\ell}$
by a bounded factor. We also let 
$$C^\infty_2(X) = \{h \in C^\infty(X): \|h \|_{2,\ell} < \infty\text{ for any }\ell = \Z_+\}.$$
 Now let $\Delta$ be a {real-valued} function on $X$, and for $z{\in\R}$ denote 
$$\pd(z) \df\mu\left(\Delta^{-1}\big([z,\infty)\big)\right).$$ 
Say that $\Delta$ is {\sl DL\/} (an abbreviation for
``distance-like'') if {there exists $z_0 \in\R$   such that  $\pd(z_0) > 0$ and
\begin{itemize}
\item[\rm (a)] $\Delta$ is  uniformly continuous on $\Delta^{-1}\big([z_0,\infty)\big)$; that is,    $\forall\,\vre>0$ there exists a neighborhood $U$ of the identity in ${G}$ such that   for any $x\in X$ with $\Delta(x) \ge z_0$,
$$  g\in U \quad\Longrightarrow \quad|\Delta(x)  - \Delta(gx) | < \vre;$$
\item[\rm (b)] the function $\pd$ 
does not decrease very fast, more precisely, if
\begin{equation}\label{dl}\tag{DL}
\exists\, c,\delta > 0
\text{ such that } 
{\pd(z ) \ge c  \pd(z - \delta)}
\ \ \ \forall\, z\ge {z_0}.
\end{equation}
\end{itemize}}
The paper \cite{KM} gives several  examples of DL functions on \hs s of semisimple Lie groups. The main goal of that paper was to study statistics of excursions of generic trajectories of   flows on $X$ into sets $\Delta^{-1}\big([z,\infty)\big)
$ for large {enough} $z$. A crucial ingredient of the argument was approximation of characteristic functions of those sets 
by smooth functions with uniformly bounded Sobolev norms. However, as was recently observed by Dubi Kelmer and Shucheng Yu, 
the argument in the main approximation statement, namely
 \cite[Lemma~4.2]{KM}, contains a mistake. To state a corrected version below,  we 
need to weaken the regularity assumption on the smooth functions approximating the sets  $\Delta^{-1}\big([z,\infty)\big)$. Namely, for $\ell\in\Z_+$ and $C >
0$, let us say that a  nonnegative function 
$h \in C^\infty_2(X) $ is {\sl $(C,\ell)$-regular\/} if
\begin{equation}\label{reg}\tag{REG}
\|h\|_{2,\ell} \le C 
\sqrt{\|h\|_1}
.
\end{equation}
Note that  the argument of  \cite{KM} used a stronger condition:
\begin{equation}\label{regold}\tag{REG-old}
\|h\|_{2,\ell} \le C  \|h\|_1
.
\end{equation}
Equivalently one can replace $\sqrt{\|h\|_1}$ in \eqref{reg} with $ \|h\|_2$, but for technical reasons it is more convenient to use the square root of the $L^1$ norm. 
Here is the corrected statement of  \cite[Lemma~4.2]{KM}:

\begin{thm}\label{lem4.2} Let  $\Delta$ be a
DL function on $X$. Then for any $\ell\in\Z_+$ there exists  $C > 0$ 
such
that for every $z {\ge z_0}$ one can find two  $(C,\ell)$-regular nonnegative
functions $h'$ and $h''$ on $X$ such that  
\eq{4.2}{
h' \le 1_{\Delta^{-1}([z,\infty))} \le h''\quad \text{and}\quad c
\pd(z)
\le \|h'\|_1 \le \|h''\|_1
\le \frac1c
\pd(z),}
with $c$ {and $z_0$} as in \eqref{dl}.
\end{thm}


Fix  a right-invariant
Riemannian metric on $G$ and the corresponding metric `dist' on $X$. For $g\in G$, let us denote by $\|g\|$ the
distance between $g\in G$ and the identity element of $G$. (Note that $\|g\| = \|g^{-1}\|$ due to the right-invariance of the metric.)
Now say  that the $G$-action on  $X$ is {\sl exponentially
mixing} 
if there exist   $\lambda ,E > 0$ and $\ell\in\Z_+$
such that for any   $\varphi, \psi \in C^\infty_2(X)$  and for any $g\in G$ one has
\begin{equation}\label{em}\tag{EM}
| \langle g\varphi ,\psi \rangle | \le E{e^{ - \lambda \|g\|}\left\| \varphi  \right\|_{2,\ell}} {\left\| \psi  \right\|_{2,\ell}}.\end{equation}
Here $\langle \cdot ,\cdot \rangle$ stands for the inner product in $L^2(X,\mu)$.

One of the main goals of \cite{KM} was, given a sequence $\{f_t: t\in\N\}$ of elements of $G$ and a sequence of non-negative functions $\{h_t: t\in\N\}$ on $X$ such that $$\sum_{t=1}^\infty \|h_t\|_1
= \infty,$$ compare the growth of $\sum_{t=1}^N  h_t(f_tx) $ for $\mu$-a.e.\ $x\in X$ with the growth of $\sum_{t=1}^N\|h_t\|_1
$ as $N\to\infty$. Results like that usually go by the name `dynamical Borel-Cantelli lemmas', see \cite{CK, HNPV}. In \cite[Proposition 4.1]{KM} such a conclusion was shown to follow from the exponential mixing of the $G$-action on $X$, the {\sl exponential divergence} of $\{f_t\}$, namely the condition
\begin{equation}\label{ed}\tag{ED}
\sup_{t\in\N}\sum_{s=1}^\infty e^{-\lambda\|f_sf_t^{-1}\|} < \infty \quad
\forall\,\lambda > 0,
\end{equation}
and the  regularity assumption \eqref{regold} on functions $\{h_t\}$. 
%

\smallskip
In the following theorem we weaken the regularity condition \eqref{regold} to \eqref{reg} and 
derive  the same conclusion:

\begin{thm}\label{prop4.1} 
Suppose that 
the $G$-action on $X$ is exponentially mixing.
Let $\{f_{t} : {t}\in \N\}$ be a sequence of elements of $G$ satisfying \eqref{ed},    and let  $ \{h_t : {t}\in \N\}$
be a  
sequence of non-negative  $(C,\ell)$-regular functions on $X$ such that $\|h_t\|_1 \le 1$ for all $t$,  and $\sum_{t=1}^\infty \|h_t\|_1
= \infty$. Then 
$$
\lim_{N\to\infty}\dfrac{\sum_{t=1}^N  h_t(f_tx)}{\sum_{t=1}^N \|h_t\|_1
} 
= 1 \quad\text{for
$\mu$-a.e.\ 
$x\in X$.}
$$
\end{thm}




Using the above theorem in place of   \cite[Proposition 4.1]{KM} and Theorem \ref{lem4.2}
 in place of \cite[Lemma  4.2]{KM}, one can then recover 
\cite[Theorem  4.3]{KM}, that is, prove

\begin{thm}\label{thm4.3} 
Suppose that 
the $G$-action on $X$ is exponentially mixing.
Let $\{f_{t} : {t}\in \N\}$ be a sequence of elements of $G$ satisfying \eqref{ed},     let  $\Delta$ be a DL function on $X$, and let {$\{r_{t} : {t}\in \N\}\subset [z_0,\infty)$} be
such that  
\eq{div}{
{\sum_{{t}=1}^{\infty} \pd(r_{t})} = \infty
.}
Then for some positive 
$c \le 1$ and for  almost
all $x\in X$ one has
$$
 c \le \liminf_{N\to\infty}\frac {\#\{ 1 \le {t}\le N\mid \Delta(f_{t}x)
\ge r_{t}\}}  {\sum_{t=1}^{N} \pd\big(r_{t}\big)} \le \limsup_{N\to\infty}\frac {\#\{ 1 \le {t}\le N\mid \Delta(f_{t}x)
\ge r_{t}\}}  {\sum_{t=1}^{N} \pd\big(r_{t}\big)} \le \frac1c.
$$
{Consequently, for any $\{r_{t} \}$ satisfying \equ{div} and almost all $x\in X$ one has $\Delta(f_tx) \ge r_t$ for infinitely many $t\in\N$. That is, in the terminology of \cite{KM}, the family of sets 
$$\{\Delta^{-1}\big([z,\infty)\big): z\in\R\}$$ is Borel-Cantelli for $\{f_{t}\}$.}
\end{thm}


 \bigskip
\noindent { {\bf Acknowledgements}. The authors are grateful to  Dubi Kelmer for bringing their attention to the mistake in \cite[Lemma  4.2]{KM}, and to Shucheng Yu for a suggestion how to correct it. Thanks are also due to Nick Wadleigh for helpful discussions.}

\section{Proofs}\label{proofs}
 Let us state a general form of Young's inequality, whose proof we give for the sake of self-containment of the paper.  Denote by $m$ the Haar measure on $G$ normalized so that the quotient map $G\to X$ locally sends $m$ to $\mu$. For $\psi\in L^1(G, m)$ and $h\in L^1(X, \mu)$, define $\psi\ast h$ by $$(\psi\ast h)(x) \df  \int_G \psi(g)h(g^{-1}x) \,dm(g).$$

 \begin{lem}\label{young}  Let $\psi\in L^1(G, m)$ and $h\in L^2(X, \mu)$. Then $\|\psi\ast h\|_2\leq \|\psi\|_1  \|h\|_2.$ \end{lem}

 \noindent{\it Proof.} We have 
  \begin{equation*}
 \begin{split}|(\psi\ast h)(x)| &\leq \int_G|\psi(g)|^{1/2}  |h(g^{-1}x)|\cdot |\psi(g)|^{1/2} \,dm(g)\\ \text{(by Cauchy-Schwarz)} &\le \left(\int_G |\psi(g)| \,dm(g)\right)^{1/2}  \left( \int_G|h(g^{-1}x)|^2   |\psi(g)| \,d m(g) \right)^{1/2}.\end{split}
 \end{equation*}
  Thus we have $$|(\psi\ast h)(x)|^2 \leq \|\psi\|_1  \int_G |\psi(g)|\cdot  |h(g^{-1}x)|^2  \,d m(g).$$ Integrating over $X$ and using Fubini's Theorem gives  
 \begin{equation*}
 \begin{split}
 \|\psi\ast h\|_2^2&\leq \|\psi\|_1  \int_X\int_G|\psi(g)|\cdot |h(g^{-1}x)|^2\,dm(g)\,d\mu(x)\\ &=\|\psi\|_1\int_G|\psi(g)|\int_X|h(g^{-1}x)|^2\,d\mu(x)\,dm(g),
 \end{split}
 \end{equation*}
which, by the $G$-invariance of $\mu$, is the same as $$\|\psi\|_1\int_G|\psi(g)|\,dm(g)\int_X|h(x)|^2\,d\mu(x)= \|\psi\|_1^2 \cdot \|h\|_2^2.\qquad\qquad\qquad \qquad\qquad\qed$$ 
 
  \begin{proof}[Proof of Theorem \ref{lem4.2}] We follow the 
proof of  \cite[Lemma  4.2]{KM}.
 For  $z\in\R$, let us
use the notation $$A(z) \df \Delta^{-1}\big([z,\infty)\big).$$ Then, for  $\vre  > 0$, let us
denote by $A'(z,\vre )$ the set of all points of $A(z)$ which are not
$\vre $-close to $\partial A(z)$, i.e. $$A'(z,\vre )\df \{x\in A(z) :
\text{dist}\big(x,\partial A(z)\big) \ge \vre \},$$ and  by $A''(z,\vre )$ the
$\vre $-neighborhood of  $A(z)$, namely
$$A''(z,\vre )\df \{x\in X : \text{dist}\big(x,
A(z)\big) \le \vre \}.$$ 
Choose {$z_0$,} $\delta$ and $c$ {as in}  (DL).   Then, using the
uniform continuity of $\Delta$ {on $\Delta^{-1}\big([z_0,\infty)\big)$}, find $\vre  > 0$ such that 
$$
|\Delta(x) -
\Delta(y)| < \delta\text{ whenever }{\Delta(x) \ge z_0 \text{ and }}\dist(x,y) < \vre .
$$
It 
follows that for all $z{\ge z_0}$, 
$$
A(z+\delta) \subset A'(z,\vre )\subset A(z)\subset A''(z,\vre )\subset
A(z-\delta);$$ therefore  one
can apply (DL) to conclude that  
\eq{4.4}{
c \mu\big( A(z)\big)\le \mu\big( A'(z,\vre )\big) \le \mu\big(
A''(z,\vre )\big) \le \frac1c\mu\big(A(z)\big).
}

Now  take a non-negative $\psi\in C^\infty(G)$ of $L^1$ norm $1$ such that
supp$\,\psi$ belongs to the ball of radius $\vre /4$ centered in $e\in G$.
 Fix  $z{\ge z_0}$ and consider   functions $h' \df \psi* 1_{A'(z,\vre /2)}$ and $h'' \df \psi* 1_{A''(z,\vre /2)}$. Then one clearly has
$$
1_{A'(z,\vre )} \le h' \le 1_{A(z)} \le h'' \le 1_{A''(z,\vre )} ,
$$
which, together with  \equ{4.4},   
 immediately implies \equ{4.2}.  It remains to choose $\ell\in\Z_+$ and find $C$ {(independent of $z$)} such that both $h'$ and $h''$ are $(C,\ell)$-regular. Take a multiindex $\alpha$ with $|\alpha| \le \ell$, and write $$\|D^\alpha h'\|_2 =  \|D^\alpha(\psi* 1_{A'(z,\vre /2)})\|_2  = \|D^\alpha(\psi)*
1_{A'(z,\vre /2)}\|_2.$$ Then, by the Young
inequality, 
$$
\|D^\alpha h'\|_2 \le \|D^\alpha(\psi)\|_1   \sqrt{\mu\big(A'(z,\vre /2)\big)}  \le
\|D^\alpha(\psi)\|_1  \sqrt {\mu\big(A(z)\big)}  \le
\|D^\alpha(\psi)\|_1\left(\frac{\|h'\|_1}{c}\right)^{1/2}  
.
$$  
Similarly, $$\|D^\alpha h''\|_2 \le \|D^\alpha(\psi)\|_1 \sqrt{\mu\big(A''(z,\vre /2)\big)} \le \|D^\alpha(\psi)\|_1  
\left(\frac{\mu\big(A(z)\big)}{c}\right)^{1/2}\le 
\|D^\alpha(\psi)\|_1  
\left(\frac{\|h''\|_1}{c}\right)^{1/2};$$
hence, with $C = \frac1{\sqrt{c}}\sum_{|\alpha|\le \ell}\|D^\alpha(\psi)\|_1$, both  $h'$ and  $h''$ are
$(C,\ell)$-regular, and the theorem is proven.
 \end{proof}

  \begin{proof}[Proof of Theorem \ref{prop4.1}] Denote $\int_X h_{t}\,d\mu = \|h_t\|_1$ by $a_t$. Following the argument in  \cite{KM}, our goal is to show that the sequence of functions $\{h_t\circ f_t\}$ satisfies a second-moment condition dating back to the work of Schmidt and 
  Sprind\v zuk: 
\begin{equation}\label{sp}\tag{SP}
\sup_{1\le M < N}\frac{\int_X\Big(\sum_{{t} = M}^{N}
h_{t}(f_tx) - \sum_{{t} = M}^{N}a_t  \Big)^2\,d\mu}{
\sum_{{t} = M}^{N}a_t} < \infty. 
\end{equation} 
the conclusion of the theorem will then follow in view of  \cite[Lemma  2.6]{KM}, which is a special case of \cite[Chapter I, Lemma 10]{Sp}.

\smallskip
Take $1 \le M < N$. As in \cite[Remark  2.7]{KM}, one can rewrite the numerator as \linebreak
$\sum_{s,t=M}^N \left(\langle f_t^{-1}h_t, f_s^{-1}h_s\rangle -a_sa_t\right)$, and then estimate it using the exponential mixing of the $G$-action on $X$: 
\begin{equation*}\begin{split}
\left|\sum_{s,t=M}^N  \langle f_t^{-1}h_t, f_s^{-1}h_s\rangle -a_sa_t\right| &\le \sum_{s,t=M}^N \left|\langle f_s f_t^{-1}h_t,  h_s\rangle -a_sa_t\right| \\
\text{(with $E,\lambda,\ell$ as in \eqref{em}) } &\le E\sum_{s,t=M}^N  {e^{ - \lambda \|f_s f_t^{-1}\|}\left\| h_t  \right\|_{2,\ell}} {\left\| h_s  \right\|_{2,\ell}}\\
\text{(by the $(C,\ell)$-regularity of $\{h_t\}$) }&\le EC^2 \sum_{s,t=M}^N   {e^{ - \lambda \|f_s f_t^{-1}\|} \sqrt{a_sa_t}}.
\end{split}\end{equation*}
Now, following an observation communicated to us by Shucheng Yu, split the above sum according to the comparison between $a_s$ and $a_t$:
\eq{3sums}{  \sum_{a_s = a_t}  e^{ - \lambda \|f_s f_t^{-1}\|} \sqrt{a_sa_t}   \ +   \sum_{a_s < a_t}  e^{ - \lambda \|f_s f_t^{-1}\|} \sqrt{a_sa_t} \   +     \sum_{a_s > a_t}  e^{ - \lambda \|f_s f_t^{-1}\|} \sqrt{a_sa_t},}
where the values of $s,t$ in the last three sums range between $M$ and $N$.  By symmetry, the last two sums are equal. Thus \equ{3sums} is not greater than
\begin{equation*}\begin{split}
  \sum_{a_s = a_t} e^{ - \lambda \|f_s f_t^{-1}\|} a_t  +   2\sum_{ a_s < a_t}  e^{ - \lambda \|f_s f_t^{-1}\|} a_t &\le 2  \sum_{s,t=M}^N e^{ - \lambda \|f_s f_t^{-1}\|} a_t \\
&\le 2  \sum_{t=M}^N a_t  \sum_{s=M}^Ne^{ - \lambda \|f_s f_t^{-1}\|}  
\\
&\le 2  \sum_{t=M}^N a_t \cdot \sup_{t\in\N}\sum_{s=1}^\infty e^{ - \lambda \|f_s f_t^{-1}\|}\ ,  \end{split}\end{equation*}
and the proof of \eqref{sp} is finished in view of \eqref{ed}.
 \end{proof}

\end{document}